\documentclass[11pt,twoside,a4paper]{article}

\usepackage{graphicx,url}
\usepackage{setspace}
\usepackage{enumerate,lineno,setspace,float}
\usepackage[small,compact]{titlesec}
\usepackage{amsmath,amsfonts,amssymb,amsthm}
\usepackage{geometry}
\usepackage{fancyhdr}
\usepackage{latexsym}
\usepackage{graphicx, float}
\usepackage{url}
\usepackage[utf8]{inputenc}
\usepackage{caption}

\long\def\delete#1{}

\input amssym.def
\input amssym.tex

\newtheorem{corollary}{Corollary}[section]

\newtheorem{definition}{Definition}[section]
\newtheorem{theorem}{Theorem}[section]

\newtheorem{example}{Example}[section]

\newtheorem{observation}{Observation}[section]

\def\diam{{\rm diam}}
\def\vp{\varphi}

\def\rn{{\rm rn}}
\def\Span{{\rm span}}
\def\ra{\rightarrow}
\def\dis{\displaystyle}

\newcommand{\be}{\begin{equation}}
\newcommand{\ee}{\end{equation}}
\newcommand{\ben}{\begin{equation*}}
\newcommand{\een}{\end{equation*}}
\newcommand{\bea}{\begin{eqnarray}}
\newcommand{\eea}{\end{eqnarray}}
\newcommand{\bean}{\begin{eqnarray*}}
\newcommand{\eean}{\end{eqnarray*}}

\begin{document}
\title{Optimal radio labelings of graphs}

\author{\textbf{Devsi Bantva}\footnote{E-mail : devsi.bantva@gmail.com (Devsi Bantva)} \\ Lukhdhirji Engineering College, Morvi - 363 642 \\ Gujarat (India)}

\pagestyle{myheadings}
\markboth{\centerline{Devsi Bantva}}{\centerline{Optimal radio labelings of graphs}}
\date{}
\openup 1.0\jot
\maketitle

\begin{abstract}
Let $\mathbb{N}$ be the set of positive integers. A radio labeling of a graph $G$ is a mapping $\vp : V(G) \rightarrow \mathbb{N} \cup \{0\}$ such that the inequality $|\vp(u)-\vp(v)| \geq \diam(G) + 1 - d(u,v)$ holds for every pair of distinct vertices $u,v$ of $G$, where $\diam(G)$ and $d(u,v)$ are the diameter of $G$ and distance between $u$ and $v$ in $G$, respectively. The radio number $\rn(G)$ of $G$ is the smallest number $k$ such that $G$ has radio labeling $\vp$ with $\max\{\vp(v) : v \in V(G)\}$ = $k$. Das \emph{et al.} [Discrete Math. $\mathbf{340}$(2017) 855--861] gave a technique to find a lower bound for the radio number of graphs. In [Algorithms and Discrete Applied Mathematics: CALDAM 2019, Lecture Notes in Computer Science $\mathbf{11394}$, springer, Cham, 2019, 161--173], Bantva modified this technique for finding an improved lower bound on the radio number of graphs and gave a necessary and sufficient condition to achieve the improved lower bound. In this paper, one more useful necessary and sufficient condition to achieve the improved lower bound for the radio number of graphs is given. Using this result, the radio number of the Cartesian product of a path and a wheel graphs is determined.\\[3mm]
 \end{abstract}

\section{Introduction}
The \emph{channel assignment problem} is the problem to assign channels to each TV or  radio transmitters such that the interference constraints are satisfied and the use of spectrum is minimized. The problem was first introduced by Hale \cite{Hale} in 1980. The interference between transmitters is closely related to geographic location of the transmitters. The closer the transmitters are, the higher the interference is and vice-versa. Hence, the frequency difference between two radio channels assigned to radio transmitters is in the inverse proportion to the distance between two transmitters. Initially only two level interference, namely \emph{high} and \emph{low}, was considered and accordingly, two transmitters are called \emph{very close} and \emph{close}, respectively. In a private communication with Griggs during 1988, Robert proposed a variation of the channel assignment problem in which close transmitters must receive different channels and very close transmitters must receive channels that are at least two apart. The problem is studied by mathematicians using graphs labeling approach.

In a graph, the transmitters are represented by vertices and two vertices are adjacent if two transmitters are very close and distance two apart if they are close. The problem of assignment of channels to transmitters is associated with graph labeling problem. Motivated through this problem, Griggs and Yeh introduced $L(2,1)$-labeling (or distance two labeling) in \cite{Griggs} as follows: An $L(2,1)$-labeling of a graph $G=(V(G),E(G))$ is a function $\vp$ from the vertex set $V(G)$ to the set of non-negative integers such that $|\vp(u)-\vp(v)| \geq 2$ if $d(u,v)=1$ and $|\vp(u)-\vp(v)| \geq 1$ if $d(u,v)=2$, where $d(u,v)$ is the distance between $u$ and $v$ in $G$. The \emph{span of $\vp$} is defined as $\Span(\vp) = \max\{|\vp(u)-\vp(v)| : u, v \in V(G)\}$. The $\lambda$-number, denoted by $\lambda(G)$, is defined as the minimum span over all $L(2,1)$-labelings of $G$. The $L(2,1)$-labeling and other distance two labeling problems have been studied by many researchers in the past two and half decades; for example, see the survey articles \cite{Calamoneri,Yeh}.

In \cite{Chartrand1,Chartrand2}, Chanrtrand \emph{et al.} extended the constraint on distance from two to the largest possible distance; the diameter of graph $G$ and introduced the concept of radio labeling as follows:
\begin{definition} A \emph{radio labeling} of a graph $G$ is a mapping $\vp : V(G) \ra \mathbb{N} \cup \{0\}$ $(\mathbb{N}$ is the set of natural numbers$)$ such that the following is satisfied for every pair of distinct vertices $u, v \in V(G)$,
\begin{equation}\label{eqn:rndef}
|\vp(u)-\vp(v)| \geq \diam(G)+1-d(u,v).
\end{equation}
The assigned integer $\vp(u)$ is called the \emph{label} of $u$ under $\vp$ and, the \emph{span of $\vp$} is defined as $$\Span(\vp) = \max\{|\vp(u)-\vp(v)| : u, v \in V(G)\}.$$ The \emph{radio number} of $G$, denoted by $\rn(G)$, is defined as
$$ \rn(G) := \dis\min_{\vp}\{\Span(\vp)\}$$
with minimum taken over all radio labelings $\vp$ of $G$. A radio labeling $\vp$ is \emph{optimal} if $\Span(\vp) = \rn(G)$.
\end{definition}

It is clear that an optimal radio labeling $\vp$ always assign 0 to some vertex and hence the span of $\vp$ is the maximum integer assigned by $\vp$. A radio labeling is a one-to-one integral function from $V(G)$ to the set of non-negative integers. Therefore, any radio labeling $\vp$ induces an ordering $O_{\vp}(V(G)):=(x_0,x_1,\ldots,x_{p-1})$ of $V(G)$ such that $0=\vp(x_0)<\vp(x_1)<\ldots<\vp(x_{p-1})=\Span(\vp)$, where $p=|V(G)|$. It is clear that if $\vp$ is an optimal radio labeling then $\Span(\vp) \leq \Span(\psi)$ for any other radio labeling $\psi$ of $G$.

A radio labeling problem is recognized as one of the tough graph labeling problems. In \cite{Chartrand1,Chartrand2}, Chartrand \emph{et al.} gave an upper bound for the radio number of paths and cycles. Liu and Zhu determined the exact radio number for paths and cycles in \cite{Zhu}. Even determining the radio number for basic graph families like paths and cycles was challenging. In \cite{Vaidya1,Vaidya2,Vaidya3}, Vaidya and Bantva determine the radio number for total graph of paths, strong product $P_2$ with $P_n$ and linear cacti. The radio number of trees remain the focus of many researchers in recent years. In \cite{Halasz}, Hal\'asz and Tuza determine the radio number of level-wise regular trees. In \cite{Li}, Li \emph{et al.} determine the radio number of complete $m$-ary trees. In \cite{Daphne1}, Liu gave a lower bound for the radio number of trees and, a necessary and sufficient condition to achieve the lower bound; the author presented a class of trees, namely spiders, achieving this lower bound. In \cite{Bantva1}, Bantva \emph{et al.} gave a different necessary and sufficient condition to achieve this lower bound and presented banana trees, firecrackers trees and a special class of caterpillars achieving this lower bound. Recently, Bantva and Liu gave a lower bound for the radio number of block graphs and three necessary and sufficient condition to achieve the lower bound in \cite{Bantva4}. They also discussed the radio number of line graph of trees and block graphs. Liu \emph{et al.} also studied the radio $k$-labeling of trees in \cite{Chavez}. In \cite{Das}, Das \emph{et al.} gave a technique to find a lower bound for the radio number of graphs. In \cite{Bantva3}, Bantva improved this technique to find a lower bound for the radio number of graphs and gave a necessary and sufficient condition to achieve the improved lower bound. Using these results, the author determined the radio number of the Cartesian product of paths and Peterson graph.

In this paper, one more useful necessary and sufficient condition to achieve the improved lower bound for the radio number of graphs given in \cite{Bantva3} is established. Some subgraphs of a given graph $G$ are characterized such that if the radio number of $G$ achieves the lower bound given in \cite{Bantva3} then these subgraphs also achieve the lower bound. Using these results, the radio number of the Cartesian product of the path graphs with wheel, star and the friendship graphs are determined.

\section{Preliminaries}

The book \cite{West} is followed for standard graph theoretic terms and notation. Only simple finite connected graphs are considered  throughout this paper. The distance $d_G(u,v)$ between two vertices $u$ and $v$ is the least length of a path joining $u$ and $v$ in a graph $G$. The suffix is dropped whenever the graph $G$ is clear in the context. The diameter of a graph $G$, denoted by $\diam(G)$, is $\max\{d_G(u,v) : u, v \in V(G)\}$. The neighborhood of $v \in G$, written as $N(v)$, is the set of vertices adjacent to $v$. Let $S \subseteq V(G)$. Define $N(S) = \{u \in V(G)\setminus S : u \mbox{ is adjacent to }v \in S\}$. The subgraph induced by $S$, denoted by $G(S)$, is a subgraph of $G$ whose vertex set is $S$ and edge set is $E(G(S)) = \{e = (u,v) \in E(G) : u, v \in S\}$. For any $u \in V(G)$, let $d_G(u,S) = \min\{d_{G}(u,v) : v \in S\}$ and $\diam(S) = \max\{d_{G}(u,v) : u, v \in S\}$. It is clear that if $|S|=1$ then $\diam(S) = 0$.

Let $H$ be an induced connected subgraph of $G$ with $\diam(H) = k$. Define layers $L_i$ of graph $G$ with respect to subgraph $H$ as follows: Set $L_0 = V(H)$ and $L_1 = N(L_0)$. Recursively define $L_{i+1} = N(L_i)$ for $1 \leq i \leq h-1$, where $$h = \max\{d_G(u,H) : u \in V(G)\},$$ which is known as the maximum level in a graph $G$. Since $G$ is a connected graph, $L_i \neq \emptyset$ for $0 \leq i \leq h$. Define the total distance of layers of graph $G$, denoted by $L(G)$, as
$$L(G) := \dis\sum_{i=1}^{h}|L_i|i.$$
For a graph $G$, define
\begin{equation*}
\delta(G) = \left\{
\begin{array}{ll}
1, & \mbox{if $|L_0| = 1$}; \\[0.2cm]
0, & \mbox{if $|L_0| \geq 2$}.
\end{array}
\right.
\end{equation*}
Let $G$ be any connected graph then for any $u,v \in V(G)$, note that the distance between $u$ and $v$ in a graph $G$ satisfies the following inequality: \begin{equation}\label{eqn:dist}
d(u,v) \leq d(u,L_0)+d(v,L_0)+\diam(L_0).
\end{equation}

In \cite{Das}, Das \emph{et al.} gave a technique to find a lower bound for the radio number of graphs. In \cite{Bantva3}, Bantva improved this technique and gave a lower bound for the radio number of graphs which is given in the following theorem.

\begin{theorem}\cite{Bantva3}\label{thm:lb} Let $G$ be a simple connected graph of order $p$, diameter $d$ and $L_0 \subseteq V(G)$. Denote $k = \diam(L_0)$ and $\delta = \delta(G)$. Then
\begin{equation}\label{eqn:lb}
\rn(G) \geq (p-1)(d-k+1)+\delta-2L(G).
\end{equation}
\end{theorem}

Although, both the lower bounds given in \cite{Das} and \cite{Bantva3} seems to be identical in notation but the difference lies in  fixing the $L_0$. In \cite{Das}, Das \emph{et al.} set a vertex or a clique of graph $G$ as $L_0$ while Bantva set all vertices of an induced subgraph $H$ of $G$ as $L_0$ with the property that two non-adjacent vertices of $V(H)$ have distance equal to $\diam(L_0)$. The readers may notice that this improved technique gives a better lower bound for the radio number of graphs, which is sharp for some classes of graph. The author of \cite{Bantva3} presented one such class of graphs, which consists of the Cartesian product of the path graph and Peterson graph in \cite{Bantva3}. In this paper, the condition to fix $L_0$ is further relaxed as follows. Set $L_0 = V(H)$, where $H$ is a connected induced subgraph of $G$ with the property that the vertices of $G$ can be ordered as $x_0,x_1,\ldots,x_{p-1}$ such that $d(x_i,x_{i+1}) = d(x_i,L_0)+d(x_{i+1},L_0)+\diam(L_0)$ for $0 \leq i \leq p-2$.

In \cite{Bantva3}, Bantva also gave a necessary and sufficient condition (given in the next theorem) to achieve the lower bound \eqref{eqn:lb} for the radio number of graphs.

\begin{theorem}\cite{Bantva3}\label{thm:main1} Let $G$ be a simple connected graph of order $p$, diameter $d$ and $L_0$ is as described earlier. Denote $k = \diam(L_0)$ and $\delta = \delta(G)$. Then
\begin{equation}
\rn(G) = (p-1)(d-k+1)+\delta-2L(G)
\end{equation}
holds if and only if there exists a radio labeling $\vp$ with $0 = \vp(x_0) < \vp(x_1) < \ldots < \vp(x_{p-1}) = \Span(\vp) = \rn(G)$ such that all the following hold for $0 \leq i \leq p-1:$
\begin{enumerate}[\rm (a)]
    \item $d(x_i,x_{i+1}) = d(x_i,L_0)+d(x_{i+1},L_0)+k$,
    \item $x_0,x_{p-1} \in L_0$ if $|L_{0}| \geq 2$ and $x_0 \in L_0$, $x_{p-1} \in L_1$ if $|L_0| = 1$,
    \item $\vp(x_0)=0$ and $\vp(x_{i+1}) =  \vp(x_i)+d+1-d(x_i,L_0)-d(x_{i+1},L_0)-k$.
\end{enumerate}
\end{theorem}

\section{Main Result}
In this Section, we give one more useful necessary and sufficient condition to achieve the improved lower bound for the radio number of graph given in \cite{Bantva3}, which rely only on the ordering of vertices of a graph.

\begin{theorem}\label{thm:main2} Let $G$ be a simple connected graph of order $p$, diameter $d \geq 2$ and $L_0$ is fixed in $G$ as described earlier. Denote $k = \diam(L_0)$ and $\delta = \delta(G)$. Then
\begin{equation}\label{rn:main}
\rn(G) = (p-1)(d-k+1)+\delta-2L(G)
\end{equation}
holds if and only if there exists an ordering $O(V(G)):=(x_0,x_1,\ldots,x_{p-1})$ of $V(G)$ such that the following conditions are satisfy.
\begin{enumerate}[\rm (a)]
    \item $d(x_0,L_0)+d(x_{p-1},L_0) = 1$ if $|L_0| = 1$ and  $d(x_0,L_0)+d(x_{p-1},L_0) = 0$ if $|L_0|\geq 2$;
    \item the distance between two vertices $x_i$ and $x_j$ $(0 \leq i < j \leq p-1)$ satisfy
\begin{equation}\label{eqn:dij}
d(x_i,x_j) \geq \dis\sum_{t=i}^{j-1}(d(x_t,L_0)+d(x_{t+1},L_0)+k-d-1)+d+1.
\end{equation}
\end{enumerate}
Moreover, under the conditions (a) and (b), the mapping $\vp$ defined by
\begin{equation}\label{eqn:vp0}
\vp(x_0) = 0,
\end{equation}
\begin{equation}\label{eqn:vp1}
\vp(x_{i+1}) = \vp(x_i) + d + 1 - d(x_i,L_0) - d(x_{i+1},L_0) - k, \; 0 \leq i \leq p-2
\end{equation}
is an optimal radio labeling of $G$.
\end{theorem}
\begin{proof}\textsf{Necessity:} Suppose that \eqref{rn:main} holds then there exists an optimal radio labeling $\vp$ of $G$ which induces an ordering $O_{\vp}(V(G)):=(x_0,x_1,\ldots,x_{p-1})$ of $V(G)$ with $0 = \vp(x_0) < \vp(x_1) < \ldots < \vp(x_{p-1}) = \Span(\vp) = \rn(G)$ such that the conditions (a)-(c) of Theorem \ref{thm:main1} hold. By Theorem \ref{thm:main1} (b), it is clear that $d(x_0,L_0)+d(x_{p-1},L_0) = 1$ when $|L_0| = 1$ and $d(x_0,L_0)+d(x_{p-1},L_0) = 0$ when $|L_0| \geq 2$. By Theorem \ref{thm:main1} (c), for any two vertices $x_i$ and $x_j$ $(j > i)$ in ordering $O_{\vp}(V(G)):=(x_0,x_1,\ldots,x_{p-1})$ of $V(G)$, we obtain
$$\vp(x_j)-\vp(x_i) = \dis\sum_{t=i}^{j-1}(d + 1 - d(x_t,L_0) - d(x_{t+1},L_0) - k).$$
Note that $\vp$ is a radio labeling of $G$ and so $\vp(x_j)-\vp(x_i) \geq d+1-d(x_i,x_{j})$. Substituting this in above equation, we obtain
\begin{equation}
d(x_i,x_j) \geq \dis\sum_{t=i}^{j-1}(d(x_t,L_0)+d(x_{t+1},L_0)+k-d-1)+d+1.
\end{equation}

\textsf{Sufficiency}:~Suppose that an ordering $O(V(G)):=(x_0, x_1, \ldots, x_{p-1})$ of $V(G)$ satisfies conditions (a)-(b) of hypothesis and $\vp$ is defined by \eqref{eqn:vp0} and \eqref{eqn:vp1}. Note that it is enough to prove that $\vp$ is a radio labeling with span equal to the right-hand side of \eqref{rn:main}. Let $x_{i}$ and $x_{j}$ ($0 \leq i < j \leq p-1$) be two arbitrary vertices then by \eqref{eqn:vp1} and using \eqref{eqn:dij}, we have
\bean
\vp(x_{j}) - \vp(x_{i}) & = & (j-i)(d+1)-\displaystyle\sum_{t=i}^{j-1}(d(x_t,L_0)+d(x_{t+1},L_0)+k) \\
& \geq & d + 1 - d(x_{i},x_{j})
\eean
and hence $\vp$ is a radio labeling. The span of $\vp$ is given by
\bean
\Span(\vp) & = & \vp(x_{p-1}) - \vp(x_{0}) \\
& = & \sum_{i=0}^{p-2} (\vp(x_{i+1}) - \vp(x_{i})) \hspace{6.5cm}
\eean
\bean
& = & \sum_{i=0}^{p-2} (d+1-d(x_i,L_0)-d(x_{i+1},L_0)-k) \\
& = & (p-1)(d+1) - \sum_{i=0}^{p-2} (d(x_i,L_0)+d(x_{i+1},L_0)+k) \\
& = & (p-1)(d-k+1) - 2 L(G) + d(x_0,L_0) + d(x_{p-1},L_0) \\
& = & (p-1)(d-k+1)+\delta-2L(G).
\eean
Therefore, $\rn(G) \leq (p-1)(d-k+1) + \delta - 2 L(G)$. This together with \eqref{eqn:lb} implies \eqref{rn:main}.
\end{proof}

A graph with no cycle is called \emph{acyclic} graph. A \emph{forest} is an acyclic graph. A \emph{tree} is a connected acyclic graph. A \emph{spanning subgraph} of a graph $G$ is a subgraph with vertex set $V(G)$. Let $H$ be a connected proper induced subgraph of a graph $G$. A \emph{spanning subgraph rooted at $H$} of a graph $G$ is a subgraph $G_H$ of $G$ with vertex set $V(G_H) = V(G)$ and $G_H(V(H)) \cong H$. A \emph{spanning tree rooted at $H$} of a graph $G$, denoted by $T_H$, is a spanning subgraph rooted at $H$ of $G$ such that $T_H \setminus H$ is a forest. A spanning tree $T_H$ rooted at $H$ is called \emph{minimum distance spanning tree rooted at $H$} if $L(T_H) = L(G)$, denoted by $T_{H}^{m}$.

\begin{observation}\label{obs2} Let $G$ be a simple connected graph of order $p$, diameter $d \geq 2$ and $L_0$ is fixed in $G$ as described earlier. Let $T_{L_0}^{m}$ be a minimum distance spanning tree rooted at $L_0$ of $G$. Then
\begin{enumerate}[\rm (a)]
    \item $\diam(T_{L_0}^{m}) = \diam(G)$,
    \item $d_{T_{L_0}^{m}}(u,L_0) = d_{G}(u,L_0)$ for all $u \in V(T^m_{L_0})$,
    \item $L(T_{L_0}^{m}) = L(G)$,
    \item $d_{T_{L_0}^{m}}(u,v) \geq d_G(u,v)$ for all $u,v \in V(G)$.
\end{enumerate}
\end{observation}

\begin{theorem}\label{thm:s1} Let $G$ be a simple connected graph of order $p$, diameter $d \geq 2$ and $L_0$ is fixed in $G$ as described earlier. Denote $\diam(L_0) = k$. If $\rn(G)$ attains a lower bound given in \eqref{eqn:lb} then $\rn(T_{L_0}^{m})$ attains a lower bound given in \eqref{eqn:lb} and $\rn(T_{L_0}^{m}) = \rn(G)$.
\end{theorem}
\begin{proof} Since $\rn(G)$ attains a lower bound given in \eqref{eqn:lb}, there exists an ordering $O(V(G))$:=($x_0$, $x_1$, $\ldots$, $x_{p-1}$) of $V(G)$ which satisfies conditions (a)-(b) of Theorem \ref{thm:main2}. Then by Observation \ref{obs2}, the same ordering of $V(T_{L_0}^{m}) = V(G)$ satisfies conditions (a)-(b) of Theorem \ref{thm:main2}. Hence, $\rn(T_{L_0}^{m})$ attains a lower bound given in \eqref{eqn:lb}. Since $V(T_{L_0}^m) = V(G)$, $\diam(T_{L_0}^m) = \diam(G)$ and $L(T_{L_0}^m) = L(G)$, it is clear that $\rn(T_{L_0}^{m}) = \rn(G)$.
\end{proof}

\begin{theorem}\label{thm:s3} Let $G$ be a simple connected graph of order $p$, diameter $d \geq 2$ and $L_0$ is fixed in $G$ as described earlier. Denote $\diam(L_0) = k$. Let $G=G_1 \supseteq G_2 \supseteq \ldots \supseteq G_t = T_{L_0}^{m}$ be a sequence of subgraphs obtained by deleting edges in $G$ to obtain $T_{L_0}^{m}$. If $\rn(G)$ attains a lower bound given in \eqref{eqn:lb} then for $1 \leq i \leq t$, $\rn(G_i)$ attains a lower bound given in \eqref{eqn:lb} and $\rn(G_i) = \rn(G)$.
\end{theorem}
\begin{proof} The proof is similar to the proof of Theorem \ref{thm:s1}.
\end{proof}

Let $G=(V(G),E(G))$ and $H=(V(H),E(H))$ be two graphs. The \emph{Cartesian product} of $G$ and $H$, denoted by $G \Box H$, is a graph with vertex set $V(G \Box H) = V(G) \times V(H)$, where two vertices $(g_1,h_1)$ and $(g_2,h_2)$ are adjacent if $g_1=g_2$ and $h_1h_2 \in E(H)$, or $h_1=h_2$ and $g_1g_2 \in E(G)$. A \emph{path} $P_m$ on $m$ vertices is a tree in which each vertex has degree at most 2. Denote the vertex set of $P_m$ by $V(P_m) = \{u_1,u_2,\ldots,u_m\}$ with $E(P_m) = \{u_iu_{i+1} : 1 \leq i \leq m-1\}$. A \emph{wheel graph} $W_n$ is a graph obtained by joining each vertex of a cycle $C_n$ to a new vertex $v_0$. Denote the vertex set of $W_n$ by $V(W_n) = \{v_0,v_1,\ldots,v_n\}$ with $E(W_n) = \{v_0v_i,v_0v_n,v_iv_{i+1},v_1v_n : 1 \leq i \leq n-1\}$. Observe that the diameter of $P_m \Box W_n$ is $m+1$.
\begin{theorem}\label{thm:pmwn} Let $m \geq 3$ and $n \geq 7$ be any integers. Then
\begin{equation}\label{pmwn:rn}
\rn(P_m \Box W_n) = \left\{
\begin{array}{ll}
\frac{1}{2}(m^2n+m^2+2m-2), & \mbox{if $m$ is even}; \\[0.2cm]
\frac{1}{2}(m^2n+m^2+2m+n-1), & \mbox{if $m$ is odd}.
\end{array}
\right.
\end{equation}
\end{theorem}
\begin{proof} We consider the following two cases.

\textsf{Case-1:} $m$ is even.

In this case, set $\{(u_{m/2},v_0),(u_{m/2+1},v_0)\}$ of $P_m \Box W_n$ as $L_0$ then $\diam(L_0) = k = 1$ and the maximum level in $P_m \Box W_n$ is $h = m/2$.

The order of $P_m \Box W_n$ and $L(P_m \Box W_n)$ are given by
\begin{equation}\label{pmwn1:p}
p := m(n+1),
\end{equation}
\begin{equation}\label{pmwn1:l}
L(P_m \Box W_n) := \frac{m}{4}(mn+2n+m-2).
\end{equation}
Substituting \eqref{pmwn1:p} and \eqref{pmwn1:l} into \eqref{eqn:lb} we obtain the right-hand side of \eqref{pmwn:rn} which is a lower bound for the radio number of $\rn(P_m \Box W_n)$. We prove that this lower bound is tight. For this purpose, we give an ordering $O(V(P_m \Box W_n)) := (x_0,x_1,\ldots,x_{p-1})$ of $V(P_m \Box W_n)$ which satisfies conditions (a) and (b) of Theorem \ref{thm:main2}.

Let $\tau$ and $\sigma$ be two permutation defined on $\{1,2,...,n\}$ as follows:
\begin{equation*}
\tau(j) = \left\{
\begin{array}{ll}
n-1, & \mbox{if $j = 1$}; \\
n, & \mbox{if $j = 2$}; \\
j-2, & \mbox{if $3 \leq j \leq n$}.
\end{array}
\right.
\end{equation*}
\begin{equation*}
\sigma(j) = \left\{
\begin{array}{ll}
\lceil j/4 \rceil, & \mbox{if $j \equiv 1$ (mod $4$)}; \\
\sum_{t=0}^{k-2}\lceil (n-t)/4 \rceil+\lceil j/4 \rceil, & \mbox{if $j \equiv k$ (mod $4$),\;$k=2,3,4$}.
\end{array}
\right.
\end{equation*}
Using these two permutations we first rename $(u_i,v_j)(1 \leq i \leq m, 0 \leq j \leq n)$ as $(a_r,b_s)$ as follows.
\begin{equation*}
(a_r,b_s) = \left\{
\begin{array}{ll}
(u_i,v_j), & \mbox{if $1 \leq i \leq m$ and $j = 0$}; \\
(u_i,v_{\sigma\tau(j)}), & \mbox{if $1 \leq i \leq m/2$ and $1 \leq j \leq n$}; \\
(u_i,v_{\sigma(j)}), & \mbox{if $m/2 < i \leq m$ and $1 \leq j \leq n$}.
\end{array}
\right.
\end{equation*}
We now define an ordering $O(V(P_m \Box W_n)):=(x_0,x_1,\ldots,x_{p-1})$ as follows: Let $x_t:=(a_r,b_s)$, where
\begin{equation*}
t := \left\{
\begin{array}{ll}
2(m/2-r)(n+1)+2s, & \mbox{if $1 \leq r \leq m/2$ and $1 \leq s \leq n$}; \\
2(m-r)(n+1)+2s-1, & \mbox{if $m/2 < r \leq m$ and $1 \leq s \leq n$}; \\
2(m/2-r)(n+1), & \mbox{if $1 \leq r \leq m/2$ and $s = 0$}; \\
2(m-r+1)(n+1)-1, & \mbox{if $m/2 < r \leq m$ and $s = 0$}.
\end{array}
\right.
\end{equation*}

Then note that $d(x_0,L_0)+d(x_{p-1},L_0)=0$. Hence, the condition (a) in Theorem \ref{thm:main2} is satisfied.

\textsf{Claim-1:} The above defined ordering $O(V(P_m \Box W_n)):=(x_0,x_1,\ldots,x_{p-1})$ satisfies \eqref{eqn:dij}.

Let $x_i$ and $x_j$ $(0 \leq i < j \leq p-1)$ be any two arbitrary vertices. Denote the right-hand side of \eqref{eqn:dij} by $E(i,j)$. Let $O(V(P_m \Box W_n)):=(x_0,x_1,\ldots,x_{p-1}) = U_0 \cup U_1 \cup \ldots \cup U_{m/2-1}$, where $U_t:=(x_{2t(n+1)},x_{(2t(n+1)+1)},\ldots,x_{2t(n+1)+2n+1})$ for $0 \leq t \leq m/2-1$. It is clear that $d(x_i,L_0)+d(x_{i+1},L_0) \leq (d+1)/2$ for all $0 \leq i \leq p-2$. Now if $x_i \in U_a, x_j \in U_b$. If $b > a+1$ then $E(i,j) < 0 < d(x_i,x_j)$. If $b=a+1$ then we consider the following two cases: (i) $j=i+2n$ and (ii) $j \neq i+2n$. If $j = i+2n$ then $d(x_i,x_j)=1$ and in this case, $E(i,j) < 0 < d(x_i,x_j)$ and if $j \neq i+2n$ then $d(x_i,x_j)=2$ and in this case, $E(i,j) \leq 2 \leq d(x_i,x_j)$. If $x_i,x_j \in U_a$ then if $x_i = x_{2t(n+1)}$ or $x_j=x_{2t(n+1)+2n+1}$ ($0 \leq t \leq m/2-1$) then $E(i,j) \leq 1 \leq d(x_i,x_j)$. If $x_{2t(n+1)} < x_i < x_j < x_{2t(n+1)+2n+1}$ then if $d(x_i,x_j)=1$ then $j-i \geq \lceil n/2 \rceil+1 \geq 3$ and hence $E(i,j) < 0 \leq d(x_i,x_j)$ and $d(x_i,x_j) \geq 2$ then $E(i,j) \leq 2 \leq d(x_i,x_j)$ which completes the proof of Claim-1.

\textsf{Case-2:} $m$ is odd.

In this case, set $\{(u_{(m+1)/2},v_0)\}$ of $P_m \Box W_n$ as $L_0$ then $\diam(L_0) = k = 0$ and the maximum level in $P_m \Box W_n$ is $h = (m+1)/2$. The order of $P_m \Box W_n$ and $L(P_m \Box W_n)$ are given by
\begin{equation}\label{pmwn2:p}
p:=m(n+1)
\end{equation}
\begin{equation}\label{pmwn2:l}
L(P_m \Box W_n) := \frac{1}{4}(m^2n+m^2+4mn-n-1).
\end{equation}
Substituting \eqref{pmwn2:p} and \eqref{pmwn2:l} into \eqref{eqn:lb} we obtain the right-hand side of \eqref{pmwn:rn} which is a lower bound for the radio number of $\rn(P_m \Box W_n)$. We prove that this lower bound is tight. For this purpose, we give an ordering $O(V(P_m \Box W_n)):=(x_0,x_1,\ldots,x_{p-1})$ of $V(P_m \Box W_n)$ which satisfies conditions (a) and (b) of Theorem \ref{thm:main2}.

Let $\tau$ and $\sigma$ are as defined earlier in Case-1. Let $\alpha$ be a permutation defined on $\{1,2,\ldots,n\}$ as follows:
\begin{equation*}
\alpha(j) = \left\{
\begin{array}{ll}
n-3, & \mbox{if $j = 1$}; \\
n-2, & \mbox{if $j = 2$}; \\
n-1, & \mbox{if $j = 3$}; \\
n, & \mbox{if $j = 4$}; \\
j-4, & \mbox{if $5 \leq j \leq n$}.
\end{array}
\right.
\end{equation*}
Using permutations $\alpha$, $\tau$ and $\sigma$, we first rename $(u_i,v_j)(1 \leq i \leq m, 0 \leq j \leq n)$ as $(a_r,b_s)$ as follows:
\begin{equation*}
(a_r,b_s) = \left\{
\begin{array}{ll}
(u_i,v_j), & \mbox{if $1 \leq i \leq m$ and $j = 0$; or $i=m$ and $1 \leq j \leq n$}; \\
(u_i,v_{\tau(j}), & \mbox{if $i = 1$ and $1 \leq j \leq n$}; \\
(u_i,v_{\alpha(j}), & \mbox{if $i = (m+1)/2$ and $1 \leq j \leq n$}; \\
(u_i,v_{\sigma\tau(j)}), & \mbox{if $2 \leq i \leq (m-1)/2$ and $1 \leq j \leq n$}; \\
(u_i,v_{\sigma(j)}), & \mbox{if $(m+3)/2 \leq i \leq m-1$ and $1 \leq j \leq n$}.
\end{array}
\right.
\end{equation*}
We now define an ordering $O(V(P_m \Box W_n)):=(x_0,x_1,\ldots,x_{p-1})$ as follows: Let $x_t:=(a_r,b_s)$, where
\begin{equation*}
t := \left\{
\begin{array}{ll}
3s-1, & \mbox{if $r = 1$ and $1 \leq s \leq n$}; \\
3s, & \mbox{if $r = (m+1)/2$ and $1 \leq s \leq n$}; \\
3s-2, & \mbox{if $r = m$ and $1 \leq s \leq n$}; \\
3n+2, & \mbox{if $r = 1$ and $s = 0$}; \\
0, & \mbox{if $r = (m+1)/2$ and $s = 0$}; \\
3n+1, & \mbox{if $r = m$ and $s = 0$}; \\
3n+2+2(r-2)(n+1)+2s, & \mbox{if $1 < r < (m+1)/2$ and $1 \leq s \leq n$}; \\
3n+2+2(r-(m+1)/2-1)(n+1)+2s-1, & \mbox{if $(m+1)/2 < r < m$ and $1 \leq s \leq n$}; \\
3n+2+2(r-1)(n+1), & \mbox{if $1 < r < (m+1)/2$ and $s = 0$}; \\
3n+2+2(r-(m+1)/2)(n+1)-1, & \mbox{if $(m+1)/2 < r < m$ and $s = 0$}.
\end{array}
\right.
\end{equation*}

Then note that $d(x_0,L_0)+d(x_{p-1},L_0) = 1$. Hence, the condition (a) in Theorem \ref{thm:main2} is satisfied.

\textsf{Claim-2:} The above defined ordering $O(V(P_m \Box W_n)):=(x_0,x_1,\ldots,x_{p-1})$ satisfies \eqref{eqn:dij}.

Let $x_i$ and $x_j$ $(0 \leq i < j \leq p-1)$ be any two arbitrary vertices. Denote the right-hand side of \eqref{eqn:dij} by $E(i,j)$. Let $O(P_m \Box W_n) := U_1 \cup \ldots \cup U_{(m-1)/2}$, where $U_1:=(x_0,x_1,\ldots,x_{3(n+1)-1})$ and $U_{t+2}:=(x_{2t(n+1)+3(n+1)},\ldots,x_{2t(n+1)+5(n+1)-1})$ for $0 \leq t \leq (m-5)/2$. Let $x_i \in U_a$ and $x_j \in U_b$. Assume $a = b = 1$. In this case, if $j \geq i+3$ then $E(i,j) \leq 0 < d(x_i,x_j)$. If $j = i+2$ then note that $d(x_i,x_j) \geq (d+2)/2$ and hence $E(i,j) \leq d/2 \leq d(x_i,x_j)$. Let $a = 1$ and $b > 1$. If $j \geq i+3$ then $E(i,j) \leq 0 < d(x_i,x_j)$. If $j = i+2$ then if $x_i = x_{3(n+1)-2}$ then note that $d(x_i,x_j) = (d-1)/2$ and $E(i,j) = (d-4)/2 < d(x_i,x_j)$. If $x_i = x_{3(n+1)-1}$ then note that $d(x_i,x_j) = 2$ and $E(i,j) = 1 < d(x_i,x_j)$. Let $a, b \geq 2$. If $b > a+1$ then $E(i,j) < 0 < d(x_i,x_j)$. If $b = a+1$ then we consider the following two cases: (i) $j = i+2n$ and (ii) $j \neq i+2n$. If $j = i+2n$ then $d(x_i,x_j) = 1$ and in this case, $E(i,j) < 0 < d(x_i,x_j)$ and if $j \neq i+2n$ then $d(x_i,x_j) = 2$ and in this case, $E(i,j) \leq 2 \leq d(x_i,x_j)$. If $x_i,x_j \in U_a$. If $x_j = x_{2t(n+1)+5(n+1)-2}$ or $x_j = x_{2t(n+1)+5(n+1)-1}$ ($0 \leq t \leq (m-5)/2$) then $E(i,j) \leq 1 \leq d(x_i,x_j)$; otherwise $E(i,j) \leq 2 \leq d(x_i,x_j)$ which completes the proof of Claim-2.
\end{proof}

An \emph{$n$-star}, denoted by $K_{1,n}$, is a tree consisting of $n$ leaves and another vertex joined to all leaves by edges. Denote the vertex set of $K_{1,n}$ by $V(K_{1,n}) = \{v_0,v_1,\ldots,v_n\}$ with $E(K_{1,n}) = \{v_0v_i : 1 \leq i \leq n\}$. A \emph{friendship graph} $F_n$ is a graph obtained by identifying one vertex of $n$ copies of cycle $C_3$ with a common vertex. Denote the vertex set of $F_n$ by $V(F_n) = \{v_0,v_1,\ldots,v_{2n}\}$ with $E(F_n) = \{v_0v_i, v_0v_{n+i},  v_{2i-1}v_{2i} : 1 \leq i \leq n\}$.

\begin{corollary}\label{thm:pmk1n} Let $m \geq 3$ and $n \geq 7$ be any integers. Then
\begin{equation*}\label{pmk1n:rn}
\rn(P_m \Box K_{1,n}) = \left\{
\begin{array}{ll}
\frac{1}{2}(m^2n+m^2+2m-2), & \mbox{if $m$ is even}; \\[0.2cm]
\frac{1}{2}(m^2n+m^2+2m+n-1), & \mbox{if $m$ is odd}.
\end{array}
\right.
\end{equation*}
\end{corollary}
\begin{proof} Observe that $P_m \Box K_{1,n}$ can be regarded as a subgraph of $P_m \Box W_n$ with identical $L_0 = \{(u_{m/2},v_0),(u_{m/2+1},v_0)\}$ when $m$ is even and $L_0 = \{(u_{(m+1)/2},v_0)\}$ when $m$ is odd and hence by Theorem \ref{thm:s3}, the radio number of $P_m \Box W_n$ and $P_m \Box K_{1,n}$ are identical.
\end{proof}

\begin{corollary}\label{thm:Fn} Let $m \geq 3$ and $n \geq 4$ be any integers. Then
\begin{equation*}\label{pmFn:rn}
\rn(P_m \Box F_{n}) = \left\{
\begin{array}{ll}
\frac{1}{2}(2m^2n+m^2+2m-2), & \mbox{if $m$ is even}; \\[0.2cm]
\frac{1}{2}(2m^2n+m^2+2m+2n-1), & \mbox{if $m$ is odd}.
\end{array}
\right.
\end{equation*}
\end{corollary}
\begin{proof} Observe that $P_m \Box F_n$ can be regarded as a subgraph of $P_m \Box W_{2n}$ with identical $L_0 = \{(u_{m/2},v_0),(u_{m/2+1},v_0)\}$ when $m$ is even and $L_0 = \{(u_{(m+1)/2},v_0)\}$ when $m$ is odd and hence by Theorem \ref{thm:s3}, the radio number of $P_m \Box W_{2n}$ and $P_m \Box F_n$ are identical.
\end{proof}

\begin{example}{\rm In Table 1, an ordering of vertices and the corresponding optimal radio labeling of $P_{7} \Box W_7$ is shown.}
\end{example}
\begin{table}[h!]\label{table2}
\centering
\caption{An ordering and optimal radio labeling for vertices of $P_{7} \Box W_7$.}
\begin{tabular}{|c|ll|ll|ll|ll|ll|ll|ll|}
\hline
$(u_{i},v_{j}) \frac{i \rightarrow}{j \downarrow}$ &  \multicolumn{2}{|c|}{1} & \multicolumn{2}{|c|}{2} &  \multicolumn{2}{|c|}{3} & \multicolumn{2}{|c|}{4} &  \multicolumn{2}{|c|}{5} & \multicolumn{2}{|c|}{6} & \multicolumn{2}{|c|}{7} \\ \hline\hline
0 & $x_{23}$ & 72 & $x_{39}$ & 139 & $x_{55}$ & 206 & \underline{\textbf{$x_{0}$}} & \underline{\textbf{\textbf{0}}} & $x_{38}$ & 133 & $x_{54}$ & 200 & $x_{22}$ & 69 \\
1 & $x_{17}$ & 51 & $x_{31}$ & 104 & $x_{47}$ & 171 & $x_{12}$ & 37 & $x_{24}$ & 76 & $x_{40}$ & 143 & $x_{1}$ & 5 \\
2 & $x_{20}$ & 60 & $x_{35}$ & 120 & $x_{51}$ & 187 & $x_{15}$ & 46 & $x_{28}$ & 92 & $x_{44}$ & 159 & $x_{4}$ & 14 \\
3 & $x_{2}$ & 6 & $x_{25}$ & 80 & $x_{41}$ & 147 & $x_{18}$ & 55 & $x_{32}$ & 108 & $x_{48}$ & 175 & $x_{7}$ & 23 \\
4 & $x_{5}$ & 15 & $x_{29}$ & 96 & $x_{45}$ & 163 & $x_{21}$ & 64 & $x_{36}$ & 124 & $x_{52}$ & 191 & $x_{10}$ & 32 \\
5 & $x_{8}$ & 24 & $x_{33}$ & 112 & $x_{49}$ & 179 & $x_{3}$ & 10 & $x_{26}$ & 84 & $x_{42}$ & 151 & $x_{13}$ & 41 \\
6 & $x_{11}$ & 33 & $x_{37}$ & 128 & $x_{53}$ & 195 & $x_{6}$ & 19 & $x_{30}$ & 100 & $x_{46}$ & 167 & $x_{16}$ & 50 \\
7 & $x_{14}$ & 42 & $x_{27}$ & 88 & $x_{43}$ & 155 & $x_{9}$ & 28 & $x_{34}$ & 116 & $x_{50}$ & 183 & $x_{19}$ & 59 \\
\hline
\end{tabular}
\end{table}

\begin{example}{\rm In Table 2, an ordering of vertices and the corresponding optimal radio labeling of $P_{8} \Box W_7$ is shown.}
\end{example}
\begin{table}[h!]\label{table3}
\centering
\caption{An ordering and optimal radio labeling for vertices of $P_{8} \Box W_7$.}
\small{
\begin{tabular}{|c|ll|ll|ll|ll|ll|ll|ll|ll|}
\hline
$(u_{i},v_{j}) \frac{i \rightarrow}{j \downarrow}$ &  \multicolumn{2}{|c|}{1} & \multicolumn{2}{|c|}{2} &  \multicolumn{2}{|c|}{3} & \multicolumn{2}{|c|}{4} &  \multicolumn{2}{|c|}{5} & \multicolumn{2}{|c|}{6} & \multicolumn{2}{|c|}{7} &\multicolumn{2}{|c|}{8} \\ \hline\hline
0 & $x_{48}$ & 201 & $x_{32}$ & 134 & $x_{16}$ & 67 & \underline{\textbf{$x_{0}$}} & \underline{\textbf{\textbf{0}}} & $x_{63}$ & 263 & $x_{47}$ & 196 & $x_{31}$ & 129 & $x_{15}$ & 62 \\
1 & $x_{56}$ & 234 & $x_{40}$ & 167 & $x_{24}$ & 100 & $x_{8}$ & 33 & $x_{49}$ & 206 & $x_{33}$ & 139 & $x_{17}$ & 72 & $x_{1}$ & 5 \\
2 & $x_{60}$ & 250 & $x_{44}$ & 183 & $x_{28}$ & 116 & $x_{12}$ & 49 & $x_{53}$ & 222 & $x_{37}$ & 155 & $x_{21}$ & 88 & $x_{5}$ & 21 \\
3 & $x_{50}$ & 210 & $x_{34}$ & 143 & $x_{18}$ & 76 & $x_{2}$ & 9 & $x_{57}$ & 238 & $x_{41}$ & 171 & $x_{25}$ & 104 & $x_{9}$ & 37 \\
4 & $x_{54}$ & 226 & $x_{38}$ & 159 & $x_{22}$ & 92 & $x_{6}$ & 25 & $x_{61}$ & 254 & $x_{45}$ & 187 & $x_{29}$ & 120 & $x_{13}$ & 53 \\
5 & $x_{58}$ & 242 & $x_{42}$ & 175 & $x_{26}$ & 108 & $x_{10}$ & 41 & $x_{51}$ & 214 & $x_{35}$ & 147 & $x_{19}$ & 80 & $x_{3}$ & 13 \\
6 & $x_{62}$ & 258 & $x_{46}$ & 191 & $x_{30}$ & 124 & $x_{14}$ & 57 & $x_{55}$ & 230 & $x_{39}$ & 163 & $x_{23}$ & 96 & $x_{7}$ & 29 \\
7 & $x_{52}$ & 218 & $x_{36}$ & 151 & $x_{20}$ & 84 & $x_{4}$ & 17 & $x_{59}$ & 246 & $x_{43}$ & 179 & $x_{27}$ & 112 & $x_{11}$ & 45 \\
\hline
\end{tabular}}
\end{table}

\section{Concluding Remarks}
In \cite{Kim}, Kim \emph{et al.} determined the radio number of Cartesian product of paths and complete graph $P_m \Box K_n$ as follows:
\begin{theorem}\cite{Kim}\label{pmkn:thm} Let $m \geq 4$ and $n \geq 3$ be integers. Then
\begin{equation}\label{pmkn:rn}
\rn(P_m \Box K_{n}) = \left\{
\begin{array}{ll}
\frac{1}{2}(m^2n-2m+2), & \mbox{if $m$ is even}; \\[0.2cm]
\frac{1}{2}(m^2n-2m+n+2), & \mbox{if $m$ is odd}.
\end{array}
\right.
\end{equation}
\end{theorem}
Theorem \ref{pmkn:thm} can also be proved using Theorem \ref{thm:main2}. The order and total level of $P_m \Box K_n$ are given by
\begin{equation}\label{pmkn:p}
p:=mn
\end{equation}
\begin{equation}\label{pmkn:l}
L(P_m \Box K_n):= \left\{
\begin{array}{ll}
\frac{1}{2}(mn(m-2)), & \mbox{if $m$ is even}; \\[0.2cm]
\frac{1}{4}((m^2-1)n), & \mbox{if $m$ is odd}.
\end{array}
\right.
\end{equation}
Substituting \eqref{pmkn:p} and \eqref{pmkn:l} in \eqref{eqn:lb} we obtained the right-hand side of \eqref{pmkn:rn} which is a lower bound for the radio number of $P_m \Box K_n$. Now it is easy to prove that the radio labeling given in \cite{Kim} satisfies conditions (a) and (b) of Theorem \ref{thm:main2} and hence the right-hand side of \eqref{eqn:lb} is exact radio number for $P_m \Box K_n$ which is \eqref{pmkn:rn} in the present case.

\section*{Acknowledgement} The author is grateful to the two anonymous referees for their careful reading of this manuscript and for their insightful and helpful comments. The research is supported by Research Promotion under Technical Education - STEM research project grant of Government of Gujarat.


\end{document}